\documentclass{amsart}
\usepackage{amssymb}
\usepackage{amsthm}
\newtheorem{theor}{Theorem}[section] \newtheorem{lem}{Lemma}[section]

\theoremstyle{definition} 
\newtheorem{ex}{Example}[section]
\theoremstyle{remark} \newtheorem{rem}{Remark}[section]
\newcommand{\pn}{\par\noindent} \newcommand{\pmn}{\par\medskip\noindent}

\begin{document}
\title{Two dynamical systems in the space of triangles}
\author{Yury Kochetkov}
\date{}
\begin{abstract} Let $M$ be the space of triangles, defined up to shifts,
rotations and dilations. We define two maps $f:M\to M$ and $g:M\to
M$. The map $f$ corresponds to a triangle of perimeter $\pi$ the
triangle with angles numerically equal to edges of the initial
triangle. The map $g$ corresponds to a triangle of perimeter
$2\pi$ the triangle with \emph{exterior} angles numerically equal
to edges of the initial triangle. For $p\in M$ the sequence
$\{p,f(p),f(f(p)),\ldots\}$ converges to the equilateral triangle
and the sequence $\{p,g(p),g(g(p)),\ldots\}$ converges to the
"degenerate triangle" with angles $(0,0,\pi)$. In Supplement an
analogous problem about inscribed-circumscribed quadrangles is
discussed.
\end{abstract}

\email{yukochetkov@hse.ru, yuyukochetkov@gmail.com} \maketitle

\section{Introduction}
\pn Dynamical systems in space of triangles are objects of an
interest for many years. For example in \cite{KS} and \cite{Un}
the map is studied that corresponds to a triangle its pedal
triangle. And in \cite{Ha} the map is studied, where a new
triangle is constructed from cevians of the given one. \pmn We
adopt another approach: we interchange roles of edges and angles.
Namely, to a triangle with perimeter $\pi$ we correspond the
triangle whose angles are numerically equal to edges of the
initial triangle, and to a triangle with perimeter $2\pi$ we
correspond the triangle whose \emph{exterior} angles are
numerically equal to edges of the initial one. \pmn Let $M$ be the
space triangles defined up to shifts, rotations and dilations.
Thus, an element of $M$ is a triple of positive numbers with the
sum $\pi$. Triples $(\alpha,\beta,\gamma)$,
$(\beta,\gamma,\alpha)$ and $(\gamma,\alpha,\beta)$ are the same,
but mirror symmetric triangles are different elements in $M$. We
denote by $\alpha$ the smallest angle of a triangle and by
$\gamma$ --- the biggest, thus, $\alpha\leqslant\beta
\leqslant\gamma$. \pmn We will consider two maps $f:M\to M$ and
$g:M\to M$. Let $p$ be a triangle of perimeter $\pi$, with angles
$\alpha,\beta,\gamma$ and let $\alpha',\beta'$ and $\gamma'$ be
lengths of edges opposite to angles $\alpha,\beta$ and $\gamma$,
respectively. Then $f(p)=q$, where $q=(\alpha',\beta', \gamma')$.
Let now $p$ be the same element in $M$, but with perimeter $2\pi$.
Let $a,b,c$ be exterior angles, adjacent to $\alpha$, $\beta$ and
$\gamma$ and $a',b',c'$ be lengths of edges, opposite to
$\alpha,\beta$ and $\gamma$, respectively. Then $g(p)=r$, where
$r=(\pi-a',\pi-b',\pi-c')$, i.e. $(a',b',c')$ are exterior angles
of the new triangle.\pmn
\begin{rem} The triangle inequality is not valid for values of
interior angles, but valid for values of exterior angles. Hence,
$f$ is not a bijection, but $g$ is.\end{rem} \pmn {\bf Theorem
2.1.} \emph{Let $p\in M$, then sequence
$\{p,f(p),f(f(p)),\ldots\}$ converges to the equilateral
triangle.} \pmn {\bf Theorem 4.1.} \emph{Let $p\in M$, then the
sequence $\{p,g(p),g(g(p)),\ldots\}$ converges to the point
$(0,0,\pi)$, which does not belong to $M$, but belong to its
boundary.}\pmn In Supplement we consider the map $h$ that
correspond to a inscribed-circumscribed quadrangle of perimeter
$2\pi$ the inscribed-circumscribed quadrangle which angles are
numerically equal to edges of the initial quadrangle. \pmn {\bf
Theorem 5.1.} \emph{Let $Q$ be an inscribed-circumscribed
quadrangle then the sequence $\{Q,h(Q),h(h(Q)),\ldots\}$ converges
to the "degenerate" quadrangle with angles $0,0,\pi,\pi$.}

\section{Properties of the map $f$}
\pn Let us remind that in a triangle the bigger edge lies opposite
the bigger angle. Thus, $\alpha'\leqslant\beta'\leqslant\gamma'$.
\begin{lem}
$$\alpha'=\frac{\pi\cdot\sin(\alpha)}{\sin(\alpha)+\sin(\beta)+\sin(\gamma)},
\quad
\beta'=\frac{\pi\cdot\sin(\beta)}{\sin(\alpha)+\sin(\beta)+\sin(\gamma)},
\quad
\gamma'=\frac{\pi\cdot\sin(\gamma)}{\sin(\alpha)+\sin(\beta)+\sin(\gamma)}.$$
\end{lem}
\begin{proof} It is enough to note that in triangle lengths of
edges are proportional to sines of opposite angles. As perimeter
must be $\pi$, it remains to find the proportionality coefficient.
\end{proof}
\begin{lem} $\alpha'\geqslant\alpha$ and equality is satisfied only when
$\alpha=\frac\pi3$,\end{lem} \begin{proof}
\begin{multline*}
\frac{\pi\cdot\sin(\alpha)}{\sin(\alpha)+\sin(\beta)+\sin(\gamma)}=
\frac{\pi\cdot\sin(\alpha)}{\sin(\alpha)+\sin(\beta)+\sin(\alpha+\beta)}=\\
=\dfrac{\pi\cdot\sin(\alpha)}
{2\sin\frac{\beta+\alpha}{2}\cos\frac{\beta-\alpha}{2}+
2\sin\frac{\beta+\alpha}{2}\cos\frac{\beta+\alpha}{2}}=
\dfrac{\pi\cdot\sin(\alpha)}{4\sin\frac{\beta+\alpha}{2}\cos\frac
\alpha 2\cos\frac \beta 2}=\dfrac{\pi\sin\frac\alpha 2}
{2\sin\frac{\beta+\alpha}{2}\cos\frac \beta 2}.\end{multline*}
If $0<x<\frac\pi6$, then $\sin(x)>\frac{3x}{\pi}$. Hence,
$$\dfrac{\pi\sin\frac\alpha 2}
{2\sin\frac{\beta+\alpha}{2}\cos\frac \beta 2}\geqslant
\dfrac{3\alpha}{4\sin\frac{\beta+\alpha}{2}\cos\frac \beta 2}\,,$$
with equality only when $\alpha=\frac\pi3$. It is enough to prove that
$$\dfrac{3\alpha}{4\sin\frac{\beta+\alpha}{2}\cos\frac \beta
2}\geqslant\alpha\Leftrightarrow 3\geqslant
4\sin\frac{\beta+\alpha}{2}\cos\frac \beta 2=2\sin\frac{\alpha+
2\beta}{2}+2\sin\frac\alpha 2.$$ Now it remains to note that the
first summand is not greater, than 2, and the second is not
greater, than 1 (because $\alpha\leqslant\frac\pi3$). \end{proof}
\begin{lem} $\gamma'\leqslant\gamma$ and equality is satisfied only when
$\gamma=\frac\pi3$. \end{lem} \begin{proof} As
$$\gamma'=\frac{\pi\cdot\sin(\gamma)}{\sin(\alpha)+\sin(\beta)+
\sin(\gamma)}=\dfrac{\pi\cdot\sin(\gamma)}
{2\sin\frac{\pi-\gamma}{2}\cdot\cos\frac{\beta-\alpha}{2}+
\sin(\gamma)}\,,$$ then $\gamma'$ is maximal when the difference
$\beta-\alpha$ is maximal. If $\gamma\geqslant\frac\pi2$, then
this difference is maximal for $\alpha=0$ and $\beta=\pi-\gamma$.
But then
$$\gamma'<\frac{\pi\cdot\sin(\gamma)} {2\sin\frac{\pi-\gamma}{2}
\cdot\cos\frac{\pi-\gamma}{2}+\sin(\gamma)}\leqslant
\frac\pi2\leqslant\gamma.$$ If $\frac\pi3\leqslant
\gamma<\frac\pi2$, then the difference is maximal when
$\beta=\gamma$ and $\alpha=\pi-2\gamma$. But then
$$\gamma'=\frac{2\pi\cdot\sin\frac\gamma2\cdot\cos\frac\gamma2}
{2\cdot\cos\frac\gamma2\cdot\cos\frac{3\gamma-\pi}{2}+
2\cdot\sin\frac\gamma2\cdot\cos\frac\gamma2}=
\frac{\pi\cdot\sin\frac\gamma2}{\sin\frac{3\gamma}{2}+
\sin\frac\gamma2}=\frac{\pi}{\cos(\gamma)+
2\cdot\cos^2\frac\gamma2+1}=\frac{\pi}{2\cos(\gamma)+2} \leqslant
\frac\pi3\leqslant\gamma.$$ \end{proof}
\begin{lem} The point $(\frac\pi3,\frac\pi3,\frac\pi3)$ is a
stationary attracting point of the map $f$. \end{lem}
\begin{proof} Let
$$\alpha=\frac\pi 3+x\varepsilon,\,\beta=\frac\pi 3+y\varepsilon,\,
\gamma=\frac\pi 3+z\varepsilon,\quad x^2+y^2+z^2=1,\quad
x+y+z=0.$$
then in the first approximation
$$\alpha'=\frac\pi 3+\frac{x\,\pi\varepsilon}{3\sqrt 3},\,
\beta'=\frac\pi 3+\frac{y\,\pi\varepsilon}{3\sqrt 3},\, \gamma'=
\frac\pi 3+\frac{z\,\pi\varepsilon}{3\sqrt 3}\,.$$  It remains to
note that $\frac{\pi}{3\sqrt 3}<1$. \end{proof} \pn The above
statements prove the theorem.
\begin{theor} Let $p\in M$, then the sequence $\{p,f(p),f(f(p)),\ldots\}$
converges to the point $(\frac\pi3,\frac\pi3,\frac\pi3)$. \end{theor}

\section{Properties of the map $g$}
\pn The reasoning here will be in terms of exterior angles. Let
$p=(\alpha,\beta,\gamma)\in M$ and $a=\pi-\alpha$, $b=\pi-\beta$
and $c=\pi-\gamma$ --- values of exterior angles. In what follows
we will use notations $a',b',c'$ instead of
$\alpha',\beta',\gamma'$ and we will assume that $a\leqslant b
\leqslant c$.
\begin{lem}
$$a'=\dfrac{2\pi\cdot\sin(a)}{\sin(a)+\sin(b)+\sin(c)},\,
b'=\dfrac{2\pi\cdot\sin(b)}{\sin(a)+\sin(b)+\sin(c)},\,
c'=\dfrac{2\pi\cdot\sin(c)}{\sin(a)+\sin(b)+\sin(c)}\,.$$
\end{lem}  \begin{proof} It is enough to mention that $\sin(\alpha)=
\sin(a)$, $\sin(\beta)=\sin(b)$ and $\sin(\gamma)=\sin(c)$. \end{proof}
\begin{lem} The point $(\frac\pi3,\frac\pi3,\frac\pi3)$ is a stationary
repelling point of the map $g$.\end{lem} \begin{proof} Let
$$a=\frac{2\pi}{3}+x\varepsilon,\quad b=\frac{2\pi}{3}+
y\varepsilon,\quad c=\frac{2\pi}{3}+z\varepsilon,\quad x+y+z=0,
\quad x^2+y^2+z^2=1.$$ Then in the first approximation
$$a'=\frac{2\pi}{3}-\frac{2\pi}{3\sqrt 3}\cdot x\varepsilon,\quad
b'=\frac{2\pi}{3}-\frac{2\pi}{3\sqrt 3}\cdot y\varepsilon,\quad
c'=\frac{2\pi}{3}-\frac{2\pi}{3\sqrt 3}\cdot z\varepsilon.$$ It remains to
note that $\frac{2\pi}{3\sqrt 3}>1$.\end{proof}  \begin{lem}
$a'\geqslant b'\geqslant c'$. \end{lem} \begin{proof}
$a<b<c\Leftrightarrow\alpha>\beta>\gamma\Leftrightarrow a'>b'>c'$.
\end{proof}

\section{Barycentric coordinates}
\pn Let us consider an equilateral triangle $\triangle ABC$ and
define a barycentric coordinates $a,b,c$, $a+b+c=2\pi$. As the
value of an exterior angle is $<\pi$, then we will work with
triangle $\triangle A_1B_1C_1$, where $A_1B_1,C_1$ are midpoints
of $BC, AC$ and $AB$, respectively.
\[\begin{picture}(180,150) \put(10,15){\line(1,0){160}}
\put(10,15){\line(2,3){80}} \put(170,15){\line(-2,3){80}}
\put(50,75){\line(1,0){80}} \put(50,75){\line(2,-3){40}}
\put(90,15){\line(2,3){40}} \put(2,10){\small A}
\put(174,10){\small B} \put(87,139){\small C}
\put(87,5){\scriptsize $C_1$} \put(135,73){\scriptsize $A_1$}
\put(38,73){\scriptsize $B_1$} \qbezier[60](10,15)(70,45)(130,75)
\qbezier[60](50,75)(110,45)(170,15)
\qbezier[60](90,130)(90,75)(90,15) \put(82,60){\scriptsize O}
\put(93,78){\scriptsize $C_2$} \put(115,45){\scriptsize $B_2$}
\put(57,45){\scriptsize $A_2$}
\end{picture}\] \begin{center}{\large Figure 4.1}\end{center} \pmn
Here points $A_1,B_1,C_1$ have coordinates $(0,\pi,\pi)$, $(\pi.0,\pi)$
and $(\pi,\pi,0)$, respectively. And points $A_2,B_2,C_2$ have
coordinates $(\pi,\frac\pi2,\frac\pi2)$, $(\frac\pi2,\pi,\frac\pi2)$
and $(\frac\pi2,\frac\pi2,\pi)$, respectively. The point $O$ --- the
center of $ABC$ has coordinates $(\frac{2\pi}{3},\frac{2\pi}{3},
\frac{2\pi}{3})$. \pmn $g$ maps
\begin{itemize} \item the triangle $A_1OC_2$ onto the triangle $A_2OC_1$:
$A_1\to A_2$, $C_2\to C_1$, $O\to O$ (and back); \item the
triangle $A_2OB_1$ onto the triangle $A_1OB_2$: $A_2\to A_1$,
$B_1\to B_2$, $O\to O$ (and back); \item the triangle $B_1OC_2$
onto the triangle $B_2OC_1$: $B_1\to B_2$, $C_2\to C_1$, $O\to O$
(and back). \end{itemize}  In what follows we will always assume
that $a\leqslant b$ and $a\leqslant c$. Let
$g(g(a,b,c))=(a'',b'',c'')$. Our aim is to prove that $a''<a$.
\pmn  Let
$$I_t=\{(a,b,c)\in \triangle A_1B_1C_1\,:\, a=t, t\leqslant \frac\pi 2\}
\quad\text{and}\quad J_t=\{(a,b,c)\in\triangle A_1B_1C_1\,:\, a=t,
\frac{2\pi}{3}>t>\frac\pi2\}.$$
\[\begin{picture}(400,170) \put(10,35){\line(1,0){160}}
\put(10,35){\line(2,3){80}} \put(170,35){\line(-2,3){80}}
\put(50,95){\line(1,0){80}} \put(50,95){\line(2,-3){40}}
\put(90,35){\line(2,3){40}} \put(2,30){\scriptsize $B_1$}
\put(174,30){\scriptsize $C_1$} \put(87,159){\scriptsize $A_1$}
\put(87,25){\scriptsize $A_2$} \put(135,93){\scriptsize $B_2$}
\put(37,93){\scriptsize $C_2$} \qbezier[60](10,35)(70,65)(130,95)
\qbezier[60](50,95)(110,65)(170,35)
\qbezier[60](90,150)(90,95)(90,35) \put(83,65){\scriptsize O}

\put(230,35){\line(1,0){160}} \put(230,35){\line(2,3){80}}
\put(390,35){\line(-2,3){80}} \put(270,95){\line(1,0){80}}
\put(270,95){\line(2,-3){40}} \put(310,35){\line(2,3){40}}
\put(222,30){\scriptsize $B_1$} \put(394,30){\scriptsize $C_1$}
\put(307,159){\scriptsize $A_1$} \put(307,25){\scriptsize $A_2$}
\put(355,93){\scriptsize $B_2$} \put(257,93){\scriptsize $C_2$}
\qbezier[60](230,35)(290,65)(350,95)
\qbezier[60](270,95)(330,65)(390,35)
\qbezier[60](310,150)(310,95)(310,35) \put(303,65){\scriptsize O}

\linethickness{0.7mm} \put(70,125){\line(1,0){40}}
\put(62,124){\scriptsize D} \put(114,124){\scriptsize E}
\put(290,85){\line(1,0){40}} \put(286,77){\scriptsize K}
\put(330,77){\scriptsize L}

\put(80,5){set $I_t$} \put(300,5){set $J_t$}
\put(92,128){\scriptsize F} \put(312,87){\scriptsize M}
\end{picture}\]
\begin{center}{\large Figure 4.2}\end{center}
\pmn Coordinates of points $D$, $E$ and $F$ are $(t,\pi-t,\pi)$,
$(t,\pi,\pi-t)$ and $(t,\pi-\frac t2,\pi-\frac t2)$, respectively.
Coordinates of points $K$, $L$ and $M$ are $(t,t,2\pi-2t)$,
$(t,2\pi-2t,t)$ and $(t,\pi-\frac t2,\pi-\frac t2)$, respectively.
\pmn We have
\[\begin{picture}(400,170) \put(10,35){\line(1,0){160}}
\put(10,35){\line(2,3){80}} \put(170,35){\line(-2,3){80}}
\put(50,95){\line(1,0){80}} \put(50,95){\line(2,-3){40}}
\put(90,35){\line(2,3){40}} \put(2,30){\scriptsize $B_1$}
\put(174,30){\scriptsize $C_1$} \put(87,159){\scriptsize $A_1$}
\put(87,25){\scriptsize $A_2$} \put(135,93){\scriptsize $B_2$}
\put(37,93){\scriptsize $C_2$} \qbezier[60](10,35)(70,65)(130,95)
\qbezier[60](50,95)(110,65)(170,35)
\qbezier[60](90,150)(90,95)(90,35) \put(83,65){\scriptsize O}

\put(230,35){\line(1,0){160}} \put(230,35){\line(2,3){80}}
\put(390,35){\line(-2,3){80}} \put(270,95){\line(1,0){80}}
\put(270,95){\line(2,-3){40}} \put(310,35){\line(2,3){40}}
\put(222,30){\scriptsize $B_1$} \put(394,30){\scriptsize $C_1$}
\put(307,159){\scriptsize $A_1$} \put(307,25){\scriptsize $A_2$}
\put(355,93){\scriptsize $B_2$} \put(257,93){\scriptsize $C_2$}
\qbezier[60](230,35)(290,65)(350,95)
\qbezier[60](270,95)(330,65)(390,35)
\qbezier[60](310,150)(310,95)(310,35) \put(303,65){\scriptsize O}

\put(83,5){$g(I_t)$} \put(303,5){$g(J_t)$} \put(93,56){\scriptsize
G} \put(261,55){\scriptsize Q} \put(354,55){\scriptsize P}
\put(312,52){\scriptsize N} \linethickness{0.7mm}
\qbezier(10,35)(90,70)(170,35) \qbezier(270,55)(310,65)(350,55)
\put(180,0){\large Figure 4.3}
\end{picture}\] \pmn Here $g(F)=G$ and
$$G=\left(\dfrac{2\pi\cdot\cos\frac t2}{\cos\frac t2+1},\,
\dfrac{\pi}{\cos\frac t2+1},\,\dfrac{\pi}{\cos\frac
t2+1}\right)=(2s,\pi-s,\pi-s),\,\text{ where }\,
s=\dfrac{\pi\cdot\cos\frac t2}{\cos\frac t2+1}\,.$$ Then $g(K)=P$,
$g(L)=Q$, $g(M)=N$ and
$$P=\left(\frac{\pi}{1-\cos(t)},\,\frac{\pi}{1-\cos(t)},
-\frac{2\pi\cdot\cos(t)}{1-\cos(t)}\right),\quad Q=\left(
\frac{\pi}{1-\cos(t)},\,-\frac{2\pi\cdot\cos(t)}{1-\cos(t)},\,
\frac{\pi}{1-\cos(t)}\right)\,.$$ \begin{rem} When $t=\frac\pi2$
(i.e. when $DE$ coincides with $B_2C_2$), then altitudes $B_1B_2$
and $C_1C_2$ are tangent to the arc $B_1GC_1$ at points $B_1$ and
$C_1$. \end{rem}
\[\begin{picture}(400,170) \put(10,35){\line(1,0){160}}
\put(10,35){\line(2,3){80}} \put(170,35){\line(-2,3){80}}
\put(50,95){\line(1,0){80}} \put(50,95){\line(2,-3){40}}
\put(90,35){\line(2,3){40}} \put(2,30){\scriptsize $B_1$}
\put(174,30){\scriptsize $C_1$} \put(87,159){\scriptsize $A_1$}
\put(87,25){\scriptsize $A_2$} \put(135,93){\scriptsize $B_2$}
\put(37,93){\scriptsize $C_2$} \qbezier[60](10,35)(70,65)(130,95)
\qbezier[60](50,95)(110,65)(170,35)
\qbezier[60](90,150)(90,95)(90,35) \put(83,65){\scriptsize O}

\put(230,35){\line(1,0){160}} \put(230,35){\line(2,3){80}}
\put(390,35){\line(-2,3){80}} \put(270,95){\line(1,0){80}}
\put(270,95){\line(2,-3){40}} \put(310,35){\line(2,3){40}}
\put(222,30){\scriptsize $B_1$} \put(394,30){\scriptsize $C_1$}
\put(307,159){\scriptsize $A_1$} \put(307,25){\scriptsize $A_2$}
\put(355,93){\scriptsize $B_2$} \put(257,93){\scriptsize $C_2$}
\qbezier[60](230,35)(290,65)(350,95)
\qbezier[60](270,95)(330,65)(390,35)
\qbezier[60](310,150)(310,95)(310,35) \put(303,65){\scriptsize O}

\linethickness{0.7mm} \qbezier(70,125)(80,135)(90,135)
\qbezier(90,135)(100,135)(110,125) \put(62,124){\scriptsize U}
\put(114,124){\scriptsize V} \put(92,125){\tiny W}
\qbezier(298,83)(310,90)(322,83) \put(292,77){\scriptsize X}
\put(325,77){\scriptsize Y} \put(312,88){\scriptsize Z}
\put(80,5){$g(g(I_t))$} \put(300,5){$g(g(J_t))$}
\put(170,0){\large Figure 4.4} \end{picture}\] We will denote by
$GG(t)$ --- the first coordinate of the point $W=g(g(F))$:
$$GG(t)=\frac{2\pi\cdot\cos(s)}{\cos(s)+1}=
\dfrac{2\pi\cdot\cos\dfrac{\pi\cdot\cos\frac t2}{\cos\frac t2+1}}
{\overset{\smallskip}{\cos\dfrac{\pi\cdot\cos\frac
t2}{\overset{\smallskip}{\cos\frac t2+1}}+1}}\,.$$ Formulas for
coordinates of points $N=g(M)$ and $Z=g(g(M))$ are the same, only
here $\frac\pi2\leqslant t\leqslant\frac{2\pi}{3}$. Thus the
function $GG$ is defined in the segment $[0,\frac{2\pi}{3}]$. In
the figure below the plot of $GG(t)$ is presented.
\[\begin{picture}(140,140) \put(10,30){\vector(1,0){120}}
\put(125,22){\small t} \put(15,15){\vector(0,1){120}}
\put(105,30){\circle*{2}} \put(15,120){\circle*{2}}
\put(100,15){$\frac{2\pi}{3}$} \put(0,115){$\frac{2\pi}{3}$}
\qbezier[50](15,30)(60,75)(105,120)
\qbezier(15,30)(45,40)(105,120) \put(40,0){Figure 4.5}
\end{picture}\]
\begin{lem}
$GG(t)<t$.\end{lem} \pn \emph{A sketch of the proof.} As $GG(0)=0$
and $GG(\frac{2\pi}{3})=\frac{2\pi}{3}$ it is enough to prove,
that $GG$ is a strictly increasing function and its plot is
downward convex. We have
$$GG'=\left(\frac{2\pi\cdot\cos(s)}{\cos(s)+1}\right)'=
-2\pi\cdot \frac{\sin(s)\cdot s'}{(\cos(s)+1)^2}=
-2\pi\cdot\frac{\sin(s)}{(\cos(s)+1)^2}\cdot
\frac{-\frac\pi2\cdot\sin\frac t2}{(\cos\frac t2+1)^2}>0$$ Then
$$GG''=-2\pi\cdot\frac{[\cos(s)\cdot(s')^2+\sin(s)\cdot
s'']\cdot(\cos(s)+1)+2\cdot\sin^2(s)\cdot(s')^2}
{(\cos(s)+1)^3}\,.$$ The numerator of this expression is
$$2\pi\cdot(1+\cos(s))\cdot[-(s')^2\cdot(2-\cos(s)-\sin(s)
\cdot s'']\,.$$ And the expression in square brackets is
$$
-\pi\cdot(2-\cos(s))\cdot\left(1-\cos\frac t2\right)+\sin(s)\cdot
\left(2-\cos\frac t2\right)\cdot\left(1+\cos\frac t2\right)\,.$$
The proof of the positivity of the above expression is a technical
task. \qed \pmn The first coordinate of the point $U$ is $\frac\pi
2\cdot (1-\cos(t))$. In the figure below are presented plots of
the first coordinates of points $U$ (the upper curve) and $W$ (the
lower curve) for $0<a\leqslant\frac\pi 2$.
\[\begin{picture}(150,150) \put(5,25){\vector(1,0){130}}
\put(10,20){\vector(0,1){120}} \put(110,25){\circle*{2}}
\put(10,125){\circle*{2}} \put(130,17){\scriptsize t}
\put(107,13){\small $\frac\pi 2$} \put(0,122){\small $\frac\pi 2$}
\qbezier[50](10,25)(60,75)(110,125) \qbezier(10,25)(60,25)(110,90)
\qbezier(10,25)(40,25)(110,125)  \put(40,0){Figure 4.6}
\end{picture}\] The first coordinate of the point $X$ is
$$\dfrac{\pi}{\overset{\vspace{1mm}}{1-\cos\dfrac{\pi}{1-\cos(t)}}}\,,
\quad \frac\pi 2<t\leqslant\frac{2\pi}{3}\,.$$ In the figure below
are presented plots of the first coordinates of points $X$ (the
upper curve) and $Z$ (the lower curve).
\[\begin{picture}(150,150) \put(5,25){\vector(1,0){150}}
\put(150,15){\small t} \put(15,15){\vector(0,1){130}}
\put(135,25){\circle*{2}} \put(130,12){\small $\frac{2\pi}{3}$}
\put(17,15){\small $\frac\pi 2$} \put(15,25){\circle*{2}}
\put(15,65){\circle*{2}} \put(5,83){\small $\frac\pi 2$}
\put(15,125){\circle*{2}} \put(2,123){\small $\frac{2\pi}{3}$}
\qbezier[60](15,85)(75,105)(135,125) \put(15,65){\line(2,1){120}}
\qbezier(15,85)(55,85)(135,125)  \put(50,0){Figure 4.7}
\end{picture}\] \begin{theor} Let $p\in M$, then the sequence
$\{p,g(p),g(g(p)),\ldots\}$ converges to the point $(0,0,\pi)$.
This point does not belong to the set $M$, but belong to its
boundary.\end{theor} \begin{proof} We see that the map $h=g(g(p))$
decreases the least exterior angle. Thus the sequence
$\{T,h(T),h(h(T)),\ldots\}$, where $T$ is a triangle, converges to
a "degenerate triangle" with angles $(0,0,\pi)$. On the other
hand, the sequence $\{g(T), h(g(T)), h(h(g(T))),\ldots\}$ also
converges to the same "degenerate triangle" (with another zero
angles). \end{proof}
\begin{ex} Let $T$ be a triangle with exterior angles
$(1,2.3,2\pi-3.3)$. Then triangles
$g(T),g(g(T)),g(g(g(T))),\ldots$ have the following exterior
angles.
\begin{align*} &(3.0300,2.6851,0.5680)\\ &(0.6418,2.5404,3.1008)\\
&(3.1217,2.9489,0.2124)\\ &(0.2953,2.8492,3.1385)\\
&(3.1408,3.1097,0.0324)\\ &(0.0673,3.0742,3.1415)\end{align*} Let
now the exterior angles of $T$ be $(1.9,2.0,2\pi-3.9)$. Here the
sequence of triples of exterior angles is of the form:
\begin{align*} &(2.3377,2.2463,1.6990)\\ &(1.8152,1.9674,2.5004)\\
&(2.4476,2.3267,1.5087)\\ &(1.6990,1.9327,2.6513)\\
&(2.5988,2.4505,1.2338)\\ &(1.5471,1.9091,2.8269)\\
&(2.7886,2.6312,0.8633)\\ &(1.3625,1.9252,2.9954)\\
&(2.9814,2.8579,0.4437)\\ &(1.1532,2.0243,3.1055) \end{align*}
\end{ex}

\section{Supplement}
\pn We will consider plane inscribed-circumscribed quadrangles
(ic-quadrangles). Sums of opposite angles of an ic-quadrangle are
$\pi$ and sums of lengths of opposite edges are equal. Up to
shifts, rotations and dilations such quadrangle is uniquely
defined by its angles and their order in going around the
quadrangle. Two angles of a ic-quadrangle are acute (and they are
adjacent) and two are obtuse (and they are also adjacent). \pmn
Let $\alpha$ and $\beta$ be obtuse angles and
$\alpha\geqslant\beta$.
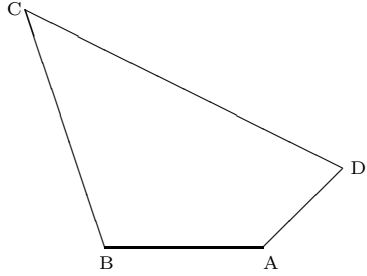
\[\begin{picture}(140,110) \put(40,10){\line(1,0){60}}
\put(40,10){\line(-1,3){30}} \put(100,10){\line(1,1){30}} \put(10
,100){\line(2,-1){120}} \put(38,2){\scriptsize B}
\put(100,2){\scriptsize A} \put(3,98){\scriptsize C}
\put(133,38){\scriptsize D} \end{picture}\] Here $\alpha$ is the
value of $\angle A$, $\beta$ is the value of $\angle B$, $\angle
C=\pi-\alpha$, $\angle D=\pi-\beta$. Let $r$ be the radius of the
inscribed circle, then
\begin{multline*}|CD|=r\cdot(\tan(\alpha/2)+\tan(\beta/2))\geqslant
|BC|=r\cdot(\tan(\alpha/2)+\cot(\beta/2))\geqslant\\
\geqslant |AB|=r\cdot(\cot(\alpha/2)+\tan(\beta/2))\geqslant
|DA|=r\cdot(\cot(\alpha/2)+\cot(\beta/2)).\end{multline*} If the
perimeter of $ABCD$ is $2\pi$, then
$$|CD|=\dfrac{\pi\cdot\sin(\frac\alpha 2)\cdot\sin(\frac\beta 2)}
{\cos(\frac\alpha 2-\frac\beta 2)}\,,\,|BC|=\dfrac{\pi\cdot
\sin(\frac\alpha 2)\cdot\cos(\frac\beta 2)}{\sin(\frac\alpha 2+
\frac\beta 2)}\,.$$ The map $h$ corresponds to an ic-quadrangle of
perimeter $2\pi$ the ic-quadrangle of perimeter $2\pi$ with angles
numerically equal to edges of initial ic-quadrangle.
\begin{theor} Let $Q$ be an ic-quadrangle of perimeter $2\pi$,
then the sequence $\{Q,h(Q),h(h(Q)),\ldots\}$ converges to a
"degenerate quadrangle" with angles $(0,0,\pi,\pi)$. \end{theor}
\pmn \emph{Sketch of the proof.} The sum of obtuse angles in the
quadrangle $h(h(Q))$ is strictly greater, than the sum of obtuse
angles in the initial quadrangle $Q$. \qed \begin{rem} The sum of
obtuse angles in $h(Q)$ can be less, than the sum of obtuse angles
in $Q$. For example, if $\alpha=1.85$ and $\beta=1,75$, then
$$\dfrac{Pi\cdot\sin(\frac\alpha 2)\cdot\sin(\frac\beta 2)}
{\cos(\frac{\alpha-\beta}{2}}+\dfrac{Pi\cdot\sin(\frac\alpha
2)\cdot \cos(\frac\beta 2)}{\sin(\frac{\alpha+\beta}{2}}=3.58<
\alpha+\beta=3.6$$\end{rem}

\vspace{5mm}

\begin{thebibliography}{XX}
\bibitem{Ha} M. Hajja, \emph{The sequence of generalized median
triangles and new shape function}, J. Geometry, {\bf 96}(2009),
71-79. \bibitem{KS} J.G. Kigston \& J.L. Synge, \emph{The sequence
of pedal triangles}, Amer. Math. Monthly, {\bf 95(7)}(1988),
609-620. \bibitem{Un} P. Ungar, \emph{Mixing property of pedal
mapping}, Amer. Math. Monthly, {\bf 97(10)}(1990), 898-900.
\end{thebibliography}
\end{document}